\documentclass[11pt]{article}
\usepackage{mathrsfs}
\usepackage{latexsym}
\usepackage{bbm}
\usepackage{amsthm}
\usepackage{amsmath}
\usepackage{amsfonts}
\usepackage{amssymb}
\usepackage{multirow}
\usepackage{latexsym}
\usepackage[dvips]{graphicx}
\usepackage{overpic}
\usepackage{graphicx}

\newtheorem{theorem}{\bf Theorem}[section]
{\vspace{2ex}}
\newtheorem{lemma}[theorem]{\bf Lemma}

\def \bS {\Bbb S}
\def \and {\, \mb
ox{\rm and}\, }

\def \supp {\,{\rm supp}\,}

\def \l {\left}
\def \r {\right}
\makeatletter

\newcommand{\Rmnum}[1]{\expandafter\@slowromancap\romannumeral #1@}

\makeatother

\begin{document}

\title{\bf Restriction  Theorem for Oscillatory Integral Operator with Certain Polynomial Phase}
\footnotetext{\footnotesize Research supported in part by the National Natural Science Foundation of China under grants 11471309, 11271162 and 11561062.}
\author{Shaozhen Xu\thanks{School of Mathematical Sciences,
        University of Chinese
        Academy of Sciences, Beijing 100049, P.R. China. E-mail address:
        {\it xushaozhen14b@mails.ucas.ac.cn}.}\,\,\,,\quad
        Dunyan Yan\thanks{School of Mathematical Sciences, University of
        Chinese Academy of Sciences, Beijing 100049, P. R. China. E-mail
        address: {\it ydunyan@ucas.ac.cn}.} }
\date{}
\maketitle{}

\begin{abstract}
 We consider the following oscillatory integral operator
\begin{equation}\label{opera-defi-1}
        T_{\alpha,m}f(x)=\int_{\mathbb R^n}e^{i(x_1^{\alpha_1} y_1^m+\cdots+x_n^{\alpha_n} y_n^m)}f(y)dy,
\end{equation}
        where the function $f$ is a Schwartz function.
        In this paper, the restriction theorem on $\mathbb{S}^{n-1}$ for this operator is obtained.
Moreover, we obtain a necessary condition which ensures the restriction theorem hold.
\end{abstract}
\textbf{Keywords}:\quad Restriction theorem,  Oscillatory integral operator,\quad $L^2$ boundedness,   Optimal estimate,   Necessary condition
\section{Introduction}
\label{}
The idea of restriction theorem for Fourier transform was due to E.M Stein in 1967. The following
classical result was obtained by Stein and Tomas \cite{tomas1975restriction} in 1975.\\

\noindent{\bf Theorem A}
    Let $\mathbb{S} \subset \mathbb{R}^n$ denote a manifold of dimension $n-1$ with nonzero Gaussian curvature and $\mathbb{S}_0$ be a compact subset of $\mathbb{S}$. Denote the measure on $\mathbb{S}$ induced by Lebesgue measure on $\mathbb{R}^n$ by $d\sigma$. Then
    \begin{equation*}
    \left(\int_{\mathbb{S}_0}|\hat{f}(\xi)|^q d\mathbf{\sigma}\right)^{1/q} \leq A(\mathbb{S}_0)  \|f\|_{L^p(\mathbb{R}^{n})}
    \end{equation*}
   for $f\in \mathscr{S}$, whenever
    $1\leq p \leq \frac{2n+2}{n+3}$   and   $q=\left(\frac{n-1}{n+1}\right)p'$,  where  $\frac{1}{p}+\frac{1}{p'}=1.$

   More details about this theorem can be found in \cite{stein1993harmonic}. Especially, the compact subset $\mathbb{S}_0$ can be $n-1$ dimensional sphere. Note that if $p=(2n+2)/(n+3)$ then $q=2$. Later, Knapp gave an example to determine the optimality of $q$ in Theorem A. On the other hand, it can be proved that no restriction theorem of any kind can hold for $f\in L^p(\mathbb{R}^n)$ when $p\geq 2n/(n+1)$, a conjecture naturally arose that the restriction theorem above extends to the range $1\leq p<2n/(n+1)$. Because of the close connection with other notable conjectures such as the Kakeya and Bochner-Riesz cojectures, as well as the local smoothing conjecture in PDE, the restriction problem has lead to tremendous interest.
    The connection can be found in \cite{wolff1999recent}, \cite{bourgain2000harmonic}, \cite{tao2004some}, \cite{laba2008harmonic}.

Actually, the operator $ T_{\alpha,m}$ can be seen as an extension of Fourier transform. Motivated by the restriction theorem of Fourier transform, we desire to get the similar result for $T_{\alpha,m}$. Now we formulate our main result.

\begin{theorem}\label{thm-1}
    In the notation \eqref{opera-defi-1}, both $\alpha_j(1\leq j\leq n)$ and $m$ are integers. Set  $\alpha_{\max}=\max\{\alpha_1, \cdots, \alpha_n\}$ and denote the the $n-1$ dimensional sphere by $\mathbb{S}^{n-1}$ . If $\alpha_j\geq 3(1\leq j\leq n), 2\leq n< m \leq n\alpha_{\max}$ and $1\leq p \leq \frac{2n\alpha_{\max}}{2n\alpha_{\max}-m}$, then the inequality
    \begin{equation}\label{main-result}
    \left(\int_{\mathbb{S}^{n-1}}|T_{\alpha,m}(f)|^2 d\mathbf{\sigma}\right)^{1/2} \leq A_{\alpha,n,m} \cdot \|f\|_{L^p(\mathbb{R}^{n})}
    \end{equation}
    holds.\\
\end{theorem}
\section{Some lemmas}
   To obtain the main result, some lemmas are needed.
\begin{lemma}{\rm(\cite{stein1993harmonic})}\label{lemma-1}
    Let $\psi$ be a  smooth function supported in the unit ball. For a real-valued function $\phi$ which satisfies
    \begin{equation*}
        |\partial_{x}^\alpha \phi|\geq 1
    \end{equation*}
    throughout the support of $\psi$  for some multi-index $\alpha$ with $|\alpha|>0$, there holds the next estimate
    \begin{equation*}
    \left|\int_{\mathbb{R}^n}e^{i\lambda \phi(x)}\psi(x)dx\right|\lesssim \lambda^{-1/k}\l(\|\psi\|_{L^\infty}+\|\nabla \psi\|_{L^1}\r)
    \end{equation*}
    where $k=|\alpha|$.
\end{lemma}

    This lemma corresponds to the higher-dimensional van der Corput Lemma.
\begin{lemma}\label{lemma-2}
    Suppose that $\mathcal{K}_a^b$ is an operator defined by
    \begin{equation*}
    \mathcal{K}_a^b f(x)=\int_{\mathbb{R}^n}\big|\|x\|_a^a-\|y\|_a^a\big|^{-\frac{n}{b}}f(y)dy,
    \end{equation*}
    where $n<a\leq b$ and $$\|x\|_a=\l(\sum_{i=1}^n|x_i|^a\r)^\frac{1}{a}.$$
   Then the next two statements hold:
\begin{itemize}
    \item[(i)] $\mathcal{K}_a^b$ is of weak type $(1,b/a)$.
    \item[(ii)] $\mathcal{K}_a^b$ maps $L^p$ to $L^q$ whenever $\frac{1}{p}=\frac{1}{q}+\frac{b-a}{n}$ and $1<p<\frac{b}{b-a}$.
\end{itemize}
\end{lemma}

   Lemma \ref{lemma-2} was essentially established in the reference \cite{shi2016sharp}. Before the further discussion, we introduce an important concept: \emph{critical point}. If a point $x_0$ satisfies $\nabla\phi(x_0)=0$, then $x_0$ is a \emph{critical point} of $\phi$.
\begin{lemma}{\rm(\cite{stein1993harmonic})}\label{lemma-3}
    Suppose $\psi$ is smooth, has compact support, and $\phi$ is a smooth real-valued function that has no critical points in the support of $\psi$. Then
    \begin{equation*}
    I(\lambda)=\int_{\mathbb{R}^n}e^{i\lambda \phi(x)}\psi(x) dx=O\l(\lambda^{-N}\r),
    \end{equation*}
    as $\lambda\rightarrow\infty$, for every $N\geq 0$.
\end{lemma}

   A critical point $x_0$ is said to be \emph{nondegenerate}, if the symmetric $n\times n$ matrix
    \begin{eqnarray*}
    \left[\frac{\partial^2\phi}{\partial x_i\partial x_j}\right](x_0)
    \end{eqnarray*}
    is invertible. In such a case, the next oscillatory integral estimate follows.

\begin{lemma}{\rm(\cite{stein1993harmonic})}\label{lemma-4}
    Suppose that $\phi$ has a nondegenerate critical point at $x_0$ and  $\phi(x_0)=0$. If $\psi$ is supported in a sufficient small neighborhood of $x_0$, then we have
    \begin{equation*}
    \left |\int_{\mathbb{R}^n}e^{i\lambda \phi(x)}\psi dx\right |\lesssim \lambda^{-n/2},
    \end{equation*}
    where $\lambda>0$.
\end{lemma}
In one-dimension case, there is a precise asymptotic expansion which is crucial in our optimality argument.
\begin{lemma}{\rm(\cite{stein1993harmonic})}\label{lemma-6}
    Suppose $k\geq 2$, and
    \begin{equation*}
    \phi(x_0)=\phi'(x_0)=\cdots=\phi^{(k-1)}(x_0)=0,
    \end{equation*}
    while $\phi^{(k)}\neq 0$. If $\psi$ is supported in a sufficient small neighborhood of $x_0$, then
    \begin{equation*}
    I(\lambda)=\int e^{i\lambda\phi(x)}\psi(x)dx\sim \lambda^{-1/k}\sum_{j=0}^{\infty}a_j\lambda^{-j/k},
    \end{equation*}
    in the sense that, for all nonnegative integers $N$ and $r$,
    \begin{equation*}
    \l(\frac{d}{d\lambda}\r)^r\l[I(\lambda)-\lambda^{-1/k}\sum_{j=0}^Na_j\lambda^{-j/k}\r]
    =O\l(\lambda^{-r-(N+1)/k}\r) \quad \text{as $\lambda\rightarrow\infty$}.
    \end{equation*}
\end{lemma}



 Before approaching Theorem \ref{thm-1}, we should establish the next lemma.
\begin{lemma}\label{lemma-5}
    Suppose that $\lambda_1, \ldots, \lambda_n$ are $n$ real numbers and   $\lambda=(\lambda_1, \ldots, \lambda_n)$. Let
    \begin{equation*}
    I(\lambda)=\int_{\mathbb{S}^{n-1}}e^{i\l(\lambda_1 y_1^{\beta_1}+\cdots+\lambda_n y_n^{\beta_n}\r)}d\sigma(y),
    \end{equation*}
    where $\mathbb{S}^{n-1}$ is the $n-1$ dimensional sphere and $\beta_j\geq 3$ are integers with $1\leq j\leq n$. Set  $\beta_{\max}=\max\{\beta_1,\cdots, \beta_n\}$.
    Then
    \begin{equation}\label{ine-1}
    \left|I(\lambda)\right|\lesssim |\lambda|^{-\frac{1}{\beta_{\max}}},
    \end{equation}
   furthermore, the estimate is optimal.
\end{lemma}
\begin{proof}
Let
$$\phi_\lambda(y)=\lambda_1y_1^{\beta_1}+\cdots+\lambda_ny_n^{\beta_n}.$$
 Choose $\psi_k$ such that
 \begin{itemize}
 \item[(1)] $\psi_k\in C^\infty(\mathbb{S}^{n-1}),$
 \item[(2)] $\supp(\psi_k)\subset U_k:=\l\{y\in \mathbb{S}^{n-1}:|y_k|\geq C_n=\frac{1}{\sqrt{n}}\r\},$\label{property-2}
     \item[(3)] $\sum_{k=1}^n\psi_k(x)\equiv 1,$
 \end{itemize}
for $1\leq k \leq n$ and $x\in \bS^{n-1}$.
Hence $\{\psi_k\}$ forms a unity partition of $\mathbb{S}^{n-1}$.
Write
\begin{align*}
I(\lambda)=\sum_{k=1}^nI_k(\lambda)&:=\sum_{k=1}^n\int_{\mathbb{S}^{n-1}}e^{i\phi_\lambda(y)}\psi_k(y)d\sigma(y)
\\
&=\sum_{k=1}^n\int_{\mathbb{S}^{n-1}}e^{i|\lambda|\phi_\xi(y)}\psi_k(y)d\sigma(y),
\end{align*}
where
$$\xi=(\xi_1,\cdots,\xi_n)=\frac{\lambda}{|\lambda|}$$
and
$$ \phi_\xi=\xi_1y_1^{\beta_1}
+\cdots+\xi_ny_n^{\beta_n}.$$

In order to prove that $I(\lambda)$ satisfies the inequality (\ref{ine-1}), it suffices to prove the same claim for every $I_k(\lambda)(1\leq k\leq n)$. Without loss of generality, we analyze $I_n(\lambda)$. A direct computation yields
\begin{align*}
I_n(\lambda)=&\int_{\mathbb{R}^{n-1}}e^{i|\lambda|\l(\xi_1 y_1^{\beta_1}+\cdots+\xi_n \l(\sqrt{1-|y'|^2} \r)^{\beta_n}\r)}\widetilde{\psi_n}(y')dy'\\
&+\int_{\mathbb{R}^{n-1}}e^{i|\lambda|\l(\xi_1 y_1^{\beta_1}+\cdots+(-1)^{\beta_n}\xi_n \l(\sqrt{1-|y'|^2} \r)^{\beta_n}\r)}
\widetilde{\psi_n}(y')dy'\\
:=&I_n^1(\lambda)+I_n^2(\lambda).
\end{align*}
  On account of the similarity between $I_n^1(\lambda)$ and $I_n^2(\lambda)$, Lemma \ref{lemma-5} is simplified to prove that $I_n^1(\lambda)$ satisfies the inequality (\ref{ine-1}).
Given $y=(y',y_n)$, then
\begin{equation*}
\nabla_{y'} \phi_{\xi}=\l(\frac{\partial \phi_{\xi}}{\partial y_1}, \cdots, \frac{\partial \phi_{\xi}}{\partial y_k}, \cdots, \frac{\partial \phi_{\xi}}{\partial y_{n-1}}\r),
\end{equation*}
where
\[\frac{\partial \phi_{\xi}}{\partial y_k}=\beta_k \xi_k y_k^{\beta_k-1}-\beta_n\xi_n y_n^{\beta_n-2}y_k.\]

 The equality $\nabla_{y'} \phi_{\xi}={\bf 0}$ brings us three kinds of critical points.\\
 \begin{itemize}
\item[(I)] $y'={\bf 0};$
\item[(II)] $\text{Each entry in } y' \text{ is nonzero};$
\item[(III)]  Else.
\end{itemize}

If all the critical points are outside the support of $\widetilde{\psi}_n$, we can derive the inequality \eqref{ine-1} from Lemma \ref{lemma-3}. So we always suppose that the critical points are contained in the support of $\widetilde{\psi}_n$.\\

We first take the critical point $y'=\bf{0}$ (naturally $y_n=1$) for example.
For this critical point, we split the argument into two cases.\\

Case I:\quad $|\xi_n|\ll 1$.\\

We denote $\xi_{k'}=\{\xi_k:|\xi_k|\geq |\xi_l|,1\leq l \leq n-1\}$. Obviously, $\l|\xi_{k'}\r|\approx 1$, then $|\xi_n|\ll |\xi_{k'}|$. Although $\xi_{k'}$ may be not unique, it does not affect our estimate. Since $|\xi_n|\ll 1$,
\begin{equation*}
\l|\frac{\partial^{\beta_{k'}}\phi_\xi}{\partial y_{k'}^{\beta_{k'}}}(0)\r|\approx\beta_{k'}!\cdot|\xi_{k'}|> |\xi_{k'}|\approx 1
\end{equation*}
throughout the support of $\widetilde{\psi}_n$, we can conclude
\begin{equation*}
|I_n^1(\lambda)|\lesssim |\lambda|^{-\frac{1}{\beta_{k'}}}\leq |\lambda|^{-\frac{1}{\beta_{\max}}}.
\end{equation*}
by employing Lemma \ref{lemma-1}\\

Case II:\quad $|\xi_n|\approx 1$.\\

We will show that in this case the critical point $y'={\bf 0}$ is nondegenerate. At the critical point, it follows
\begin{equation*}
\begin{cases}
\frac{\partial^2 \phi_\xi}{\partial y_k^2}=\beta_k(\beta_k-1)\xi_k y_k^{\beta_k-2}-\beta_n\xi_ny_n^{\beta_n-2}+\beta_n(\beta_n-2)\xi_n
y_n^{\beta_n-4}y_k^2=-\beta_n\xi_n,\\
\frac{\partial^2 \phi_\xi}{\partial y_k\partial y_l}=\beta_n(\beta_n-2)\xi_ny_n^{\beta_n-4}y_ky_l=0, \quad k\neq l.
\end{cases}
\end{equation*}
Then the matrix
\begin{eqnarray*}
\left[\frac{\partial^2\phi_\xi}{\partial y_i\partial y_j}\right](0)=\left(
  \begin{array}{ccccc}
    -\beta_n\xi_n & \cdots & 0 &\cdots &0\\
    \vdots & \ddots& & &\vdots\\
    0& &-\beta_n\xi_n & &0\\
    \vdots&  & &\ddots &\vdots \\
    0& \cdots& 0& \cdots&-\beta_n\xi_n
  \end{array}
\right)
\end{eqnarray*}
is invertible obviously. Thus the critical point $y'=0$ is nondegenerate, Lemma \ref{lemma-4} implies
\begin{equation*}
|I_n^1(\lambda)|\lesssim |\lambda|^{-\frac{n-1}{2}}\leq |\lambda|^{-\frac{1}{\beta_{\max}}}.
\end{equation*}\\

We now turn to the second kind of critical points of $\phi_\xi$. Similar to the argument above, we also divide this argument into two cases.\\

Case I:\quad $|\xi_n|\ll 1$.\\

For convenience, we still use the notations we have denoted above. In this case $|\xi_n|\ll 1$ and $C_n\leq y_n<1$, naturally
\begin{equation*}
\l|\frac{\partial^{\beta_{k'}}\phi_\xi}{\partial y_{k'}^{\beta_{k'}}}\r|\approx\beta_{k'}!\cdot|\xi_{k'}|> |\xi_{k'}|\approx 1
\end{equation*}
still holds throughout the support of $\widetilde{\psi}_n$. Lemma \ref{lemma-1} indicates
\begin{equation*}
|I_n^1(\lambda)|\lesssim |\lambda|^{-\frac{1}{\beta_{k'}}}\leq |\lambda|^{-\frac{1}{\beta_{\max}}}.
\end{equation*}

Case II:\quad $|\xi_n|\approx 1$.\\

Similarly, We shall show that in this case the critical point $y'$ is nondegenerate. At the critical point, we have
\begin{equation*}
\begin{cases}
\frac{\partial^2 \phi_\xi}{\partial y_k^2}=\beta_k(\beta_k-1)\xi_k y_k^{\beta_k-2}-\beta_n\xi_ny_n^{\beta_n-2}+\beta_n(\beta_n-2)\xi_n
y_n^{\beta_n-4}y_k^2,\\
\frac{\partial^2 \phi_\xi}{\partial y_k\partial y_l}=\beta_n(\beta_n-2)\xi_ny_n^{\beta_n-4}y_ky_l, \quad k\neq l.
\end{cases}
\end{equation*}
Provided that the critical point satisfies $\beta_k \xi_k y_k^{\beta_k-1}-\beta_n\xi_n y_n^{\beta_n-2}y_k=0$ and $y_k\neq 0$, we have $\beta_k \xi_k y_k^{\beta_k-2}=\beta_n\xi_n y_n^{\beta_n-2}$. Based on this, it is easy to verify
\begin{align*}
\frac{\partial^2 \phi_\xi}{\partial y_k^2}&=\xi_n\beta_n(\beta_k-1)y_n^{\beta_n-2}+\xi_n\beta_n
(\beta_n-2)y_n^{\beta_n-4}y_k^2-\xi_n\beta_ny_n^{\beta_n-2}\\
&=\xi_n\beta_n(\beta_k-2)y_n^{\beta_n-2}+\xi_n\beta_n
(\beta_n-2)y_n^{\beta_n-4}y_k^2\\
&=\xi_n\beta_ny_n^{\beta_n-2}\l[(\beta_k-2)+(\beta_n-2)
\l(\frac{y_k}{y_n}\r)^2\r]\\
&=\xi_n\beta_n(\beta_n-2)y_n^{\beta_n-2}\l[\frac{\beta_k-2}{\beta_n-2}+
\l(\frac{y_k}{y_n}\r)^2\r]
\end{align*}
and
\begin{equation*}
\frac{\partial^2 \phi_\xi}{\partial y_k\partial y_l}=\xi_n\beta_n(\beta_n-2)y_n^{\beta_n-2}\l(\frac{y_k}{y_n}\r)\l(\frac{y_l}{y_n}\r).
\end{equation*}
Then the matrix $\left[\frac{\partial^2\psi_\xi}{\partial y_i\partial y_j}\right](y')$ is
\begin{align*}
\xi_n\beta_n(\beta_n-2)y_n^{\beta_n-2}\left(
  \begin{array}{ccccc}
    \l[\frac{\beta_1-2}{\beta_n-2}+
\l(\frac{y_1}{y_n}\r)^2\r] & \cdots & \l(\frac{y_k}{y_n}\r)\l(\frac{y_1}{y_n}\r) &\cdots &\l(\frac{y_{n-1}}{y_n}\r)\l(\frac{y_1}{y_n}\r)\\
    \vdots & \ddots& & \ddots&\vdots\\
   \l(\frac{y_1}{y_n}\r)\l(\frac{y_k}{y_n}\r)& &\l[\frac{\beta_k-2}{\beta_n-2}+
\l(\frac{y_k}{y_n}\r)^2\r]& &\l(\frac{y_{n-1}}{y_n}\r)\l(\frac{y_k}{y_n}\r)\\
    \vdots& \ddots & &\ddots &\vdots \\
    \l(\frac{y_1}{y_n}\r)\l(\frac{y_{n-1}}{y_n}\r)& \cdots& \l(\frac{y_k}{y_n}\r)\l(\frac{y_{n-1}}{y_n}\r)& \cdots&\l[\frac{\beta_{n-1}-2}{\beta_n-2}+
\l(\frac{y_{n-1}}{y_n}\r)^2\r]
  \end{array}
\right).
\end{align*}
We shall show that the matrix is invertible. Decompose this matrix into two parts:
\begin{equation*}
\left[\frac{\partial^2\psi_\xi}{\partial y_i\partial y_j}\right](y')=\xi_n\beta_n(\beta_n-2)y_n^{\beta_n-2}\l(\Lambda+Y^T\cdot Y\r)
\end{equation*}
where
\begin{equation*}
\Lambda=\left(
  \begin{array}{ccccc}
    \frac{\beta_1-2}{\beta_n-2} & \cdots & 0 &\cdots &0\\
    \vdots & \ddots& & &\vdots\\
    0& &\frac{\beta_k-2}{\beta_n-2} & &0\\
    \vdots&  & &\ddots &\vdots \\
    0& \cdots& 0& \cdots&\frac{\beta_{n-1}-2}{\beta_n-2}
  \end{array}
\right),
\end{equation*}
As $\Lambda+Y^T\cdot Y$ is positive definite, 
the critical point $y'$ is nondegenerate apparently.
 By a smooth truncation of $\widetilde{\psi}_n$ (supported in such a small neighborhood of $y'$ that the neighborhood contains only one critical point $y'$), Lemma \ref{lemma-3} leads to
\begin{equation*}
\l|I_n^1(\lambda)\r|\lesssim |\lambda|^{-\frac{n-1}{2}}\leq |\lambda|^{-\frac{1}{\beta_{\max}}}.
\end{equation*}

Combining the arguments of the first and the second kind of critical points gives the proof of the third kind of critical points. The verification is easy and we omit it.\\
Up to now, we have confirmed the estimate. For the claim in \ref{lemma-5}, it remains to give the optimality argument.\\
To prove the estimate is optimal, we need only show that in a fixed direction of $\lambda=\l(\lambda_1,\cdots,\lambda_n\r)$ denoted by $\widetilde{\lambda}$, it follows
\begin{equation}\label{equality-1}
\l|I\l(\widetilde{\lambda}\r)\r|=a_0\l|\widetilde{\lambda}\r|^{-1/\beta_{\max}}+
o\left(\l|\widetilde{\lambda}\r|^{-1/\beta_{\max}}\right)\quad \ \text{as $\l|\widetilde{\lambda}\r|\rightarrow\infty$ }.
\end{equation}
where $a_0>0$ is a constant not related to $\widetilde{\lambda}$ and $o$ denote the higher-order term. Without loss of generality, set $\beta_n=\beta_{\max}$, we will show that for $\widetilde{\lambda}=\left(0,\cdots,0,\left|\widetilde{\lambda}\right|\right)
=\left|\widetilde{\lambda}\right|(0,0,\cdots,1)$, the equality (\ref{equality-1}) holds.\\

In this case,
\begin{align*}
I\l(\widetilde{\lambda}\r)&=\int_{\mathbb{S}^{n-1}}e^{i\l|\widetilde{\lambda}\r| y_n^{\beta_n}}d\sigma(y)\\
&=\int_{\mathbb{S}^{n-1}}e^{i\l|\widetilde{\lambda}\r| y_n^{\beta_n}}\psi_n(y)d\sigma(y)+\int_{\mathbb{S}^{n-1}}e^{i\l|\widetilde{\lambda}\r| y_n^{\beta_n}}\l(1-\psi_n(y)\r)d\sigma(y)\\
&=I_n\l(\widetilde{\lambda}\r)+I_n^c\l(\widetilde{\lambda}\r)
\end{align*}
No matter whatever critical points $I_n\l(\widetilde{\lambda}\r)$ have, we can derive
\begin{equation}\label{equality-2}
\l|I_n\l(\widetilde{\lambda}\r)\r|\lesssim \l|\widetilde{\lambda}\r|^{-\frac{n-1}{2}}.
\end{equation}
from Case II of the three kinds of the critical points we have discussed.
Obviously,
\begin{equation}\label{equality-4}
\l|\widetilde{\lambda}\r|^{-\frac{n-1}{2}}=
o\l(\l|\widetilde{\lambda}\r|^{-1/\beta_{n}}\r).
\end{equation}

On the other hand, the change of variable formula yields
\begin{align*}
I_n^c\l(\widetilde{\lambda}\r)&=\int_{-1}^{1}e^{i\l|\widetilde{\lambda}\r| y_n^{\beta_n}}\int_{\sqrt{1-y_n^2}
\mathbb{S}^{n-2}}\frac{\l(1-\psi_n(y_n,\theta)\r)}{\sqrt{1-y_n^2}}d\theta dy_n.
\end{align*}
Define \[\Psi_n(y_n):=\int_{\sqrt{1-y_n^2}
\mathbb{S}^{n-2}}\frac{\l(1-\psi_n(y_n,\theta)\r)}{\sqrt{1-y_n^2}}d\theta.\]
Then
\begin{align*}
I_n^c\l(\widetilde{\lambda}\r)&=\int_{-1}^{1}e^{i\l|\widetilde{\lambda}\r| y_n^{\beta_n}}\Psi_n(y_n) dy_n.
\end{align*}
We claim that $\Psi_n(y_n)$ is a smooth function with compact support. The fact that $\Psi_n(y_n)$ is smooth on the interval $(-1,1)$ is obvious, it suffices to verify $\Psi_n(y_n)$ is supported in $[-1,1]$. Since the equality $\sum_{k=1}^n\psi_k(x)\equiv 1$ and $U'_n=\{y\in \mathbb{S}^{n-1}:|y_n|\geq \sqrt{\frac{n-1}{n}}\}\subset U_n$ satisfies
\begin{equation*}
U'_n\cap U_k=\emptyset, \quad 1\leq k\leq n-1,
\end{equation*}
then $\psi_n(y)\equiv 1$ on $U'_n$. Thus we have
\begin{equation*}
\Psi_n(y_n)\equiv 0 \quad \ \text{when} \quad |y_n|\geq \sqrt{\frac{n-1}{n}}.
\end{equation*}
Clearly, $\Psi_n(y_n)$ is a smooth function with compact support.

By using Lemma \ref{lemma-6} in which $r=0, N=1$, we acquire
\begin{align}\label{equality-3}\notag
I_n^c\l(\widetilde{\lambda}\r)&=a_0\l|\widetilde{\lambda}\r|^{-1/\beta_{n}}+
o\left(\l|\widetilde{\lambda}\r|^{-1/\beta_{n}}\right)\\
&=a_0\l|\widetilde{\lambda}\r|^{-1/\beta_{\max}}+
o\left(\l|\widetilde{\lambda}\r|^{-1/\beta_{\max}}\right).
\end{align}
The value of $a_0$ can be computed according to the proof of Lemma \ref{lemma-6} presented in \cite{stein1993harmonic}, we figure it out that
\begin{align*}
a_0&\approx|
\Psi_n(0)|\\
&=\l|
\int_{\mathbb{S}^{n-2}}\left(1-\psi_n(0,\theta)\right)d\theta\r|.
\end{align*}
Considering the property of $\psi_n$ that $\supp(\psi_n)\subset \{y\in \mathbb{S}^{n-1}:|y_n|\geq C_n=\frac{1}{\sqrt{n}}\}$, we have $\psi_n(0,\theta)=0$ and
\begin{equation*}
a_0\approx
\left|\mathbb{S}^{n-2}\right|
\neq 0.
\end{equation*}
Combine the estimates (\ref{equality-2}), (\ref{equality-4}) and (\ref{equality-3}), we obtain (\ref{equality-1}). The proof is complete.
\end{proof}

  This lemma is an estimate of the decay rate of oscillatory integral, it plays an important role in the proof of Theorem \ref{thm-1}.

\section{Proof of Theorem \ref{thm-1}}

 For any $f\in \mathscr{S}$, the $L^2$ norm of $T_{k,m}(f)$ on $n-1$ dimensional sphere is
\begin{eqnarray*}
    \left\|T_{\alpha,m}(f)\right\|^2_{L^2(\mathbb{S}^{n-1},d\sigma)} &=& \left\langle T_{\alpha,m}(f),T_{\alpha,m}(f)\right\rangle_{\mathbb{S}^{n-1}}\\
    &=&\int_{\mathbb{S}^{n-1}}T_{\alpha,m}(f)\cdot\overline{T_{\alpha,m}(f)}d\sigma\\
    &=&\int_{\mathbb{S}^{n-1}}\int_{\mathbb R^n}e^{i(x_1^{\alpha_1} y_1^m+\cdots+x_n^{\alpha_n} y_n^m)}f(y)dy\overline{\int_{\mathbb R^n}e^{i(x_1^{\alpha_1} z_1^m+\cdots+x_n^{\alpha_n} z_n^m)}f(z)dz}d\sigma\\
    &=&\int_{\mathbb R^n}\int_{\mathbb R^n}\int_{\mathbb{S}^{n-1}}e^{i\left[(x_1^{\alpha_1} (y_1^m-z_1^m)+\cdots+x_n^{\alpha_n}( y_n^m-z_n^m)\right]}d\sigma f(y)\overline{f(z)}dydz\\
    &\leq& \int_{\mathbb R^n}\int_{\mathbb R^n}\left|\int_{\mathbb{S}^{n-1}}e^{i\left[(x_1^{\alpha_1} (y_1^m-z_1^m)+\cdots+x_n^{\alpha_n}( y_n^m-z_n^m)\right]}d\sigma \right| \left |f(y)\overline{f(z)}\right |dydz.
\end{eqnarray*}
    By means of Lemma \ref{lemma-5}, the oscillatory integral in the last line satisfies
\begin{equation*}
    \left|\int_{\mathbb{S}^{n-1}}e^{i\left [(x_1^{\alpha_1} (y_1^m-z_1^m)+\cdots+x_n^{\alpha_n}( y_n^m-z_n^m)\right ]}d\sigma \right |\lesssim \left| \sum\limits_{i=1}^n\left(\left|y_i^m-z_i^m\right|^2\right)^{\frac{1}{2}}\right|^{-\frac{1}{\gamma}}
\end{equation*}
    whenever $\gamma\geq \alpha_{\max}$.
    By the equivalence of the norms in finite dimensional linear normed space $(\text{i.e.} \|\|_1 \approx \|\|_2)$, we obtain
    \begin{equation*}
    \sum\limits_{i=1}^n\left(\left|y_i^m-z_i^m\right|^2\right)^{\frac{1}{2}}\approx
    \sum\limits_{i=1}^n\left|y_i^m-z_i^m\right|
    \end{equation*}
    and the simple inequality
    \begin{equation*}
    \sum\limits_{i=1}^n\left|y_i^m-z_i^m\right|\geq \left|\|y\|_m^m-\|z\|_m^m\right|
    \end{equation*}
    where the notation $\|\cdot\|_m^m$ is same as what we have denoted in Lemma \ref{lemma-2}.\\

    Then
    \begin{equation*}
    \left|\int_{\mathbb{S}^{n-1}}e^{i\left [(x_1^{\alpha_1} (y_1^m-z_1^m)+\cdots+x_n^{\alpha_n}( y_n^m-z_n^m)\right ]}d\sigma \right |\lesssim \left|\|y\|_m^m-\|z\|_m^m\right|^{-\frac{1}{\gamma}}.
    \end{equation*}
    By employing Lemma \ref{lemma-2}, the next inequality holds
    \begin{align*}
    \left\|T_{\alpha,m}(f)\right\|_{L^2(\mathbb{S}^{n-1},d\sigma)} &\lesssim
    \int_{\mathbb R^n}\int_{\mathbb R^n}\left|\|y\|_m^m-\|z\|_m^m\right|^{-\frac{1}{\gamma}} \left |f(y)\overline{f(z)}\right |dydz\\
    &= \int_{\mathbb {R}^n}\mathcal{K}_m^{n\gamma}(|f|)(z)|f(z)|dz.
    \end{align*}

   If the right side of the inequality above is finite, then we get the $L^2$ boundedness on $n-1$ dimensional sphere of the operator $T_{\alpha,m}$. For this purpose, we require that the operator $\mathcal{K}_m^{n\gamma}$ map $L^p$ to $L^{p'}$ where $p'$ is the conjugate exponential of $p$ (i.e. $1/p+1/p'=1$) and $p>1$. Consequently, we have $p=\frac{2n\gamma}{2n\gamma-m}$. Considering $\gamma\geq \alpha_{\max}$, we can extend the index $p=\frac{2n\gamma}{2n\gamma-m}$ to $1< p\leq\frac{2n\alpha_{\max}}{2n\alpha_{\max}-m}$. The case $p=1$ is trivial, we omit its argument. Hence the statements in Theorem \ref{thm-1} are proved.

\section{Characterization of necessary condition}
In this section, we mainly investigate a necessary condition which ensures the inequality (\ref{main-result}) hold. Motivated by the optimality argument of the restriction theorem in \cite{stein1993harmonic} and the Knapp's example stated in \cite{rahmanrestriction}, we give a necessary condition. For other applications of this technique, readers may refer to \cite{yang2004sharp} and \cite{phong1997newton}.\\

\begin{theorem}
If the inequality (\ref{main-result}) holds, the index $p$ must satisfy
\begin{equation*}
1\leq p \leq \frac{2(2+\min\limits_{i}\sum\limits_{j\neq i}^n\alpha_j)}{2(2+\min\limits_{i}\sum\limits_{j\neq i}^n\alpha_j)-(n-1)m}.
\end{equation*}
\end{theorem}
\begin{proof}
Owing to the smoothness of $T_{\alpha,m}(f)$, the inequality of (\ref{main-result}) actually equals
\begin{equation*}
\int_{||x|-1|\leq \delta^2}\l|T_{\alpha,m}(f)\r|^2dx \leq A_{\alpha,n,m} \delta^2 \|f\|_{L^p(\mathbb{R}^{n})}^2,
\end{equation*}
for sufficiently small positive $\delta(\delta<\frac{1}{2})$. Since the shell $\{x:1-\delta^2\leq |x|\leq 1+\delta^2\}$ contains the following rectangle
\[R_\delta:=\{x=(x_1,\cdots,x_n):|x_1-1|\leq c\delta^2, |x_i|\leq c\delta, 2\leq i\leq n\},\]
if $c$ is a sufficiently small constant, we obtain
\begin{equation}\label{inequality-2}
\int_{R_\delta}\l|T_{\alpha,m}(f)\r|^2dx \leq \int_{||x|-1|\leq \delta^2}\l|T_{\alpha,m}(f)\r|^2dx \leq A_{\alpha,n,m} \delta^2 \|f\|_{L^p(\mathbb{R}^{n})}^2.
\end{equation}
In this case, we set
\begin{equation*}
f(y)=\prod_{j=1}^nf_j(y_j);
\end{equation*}
where
\begin{align*}
f_1(y_1)&=e^{-iy_1^m}\chi_{E_1}(y_1), \\
f_j(y_j)&=\chi_{E_j}(y_j), \quad (2\leq j\leq n)
\end{align*}
in which $E_1:=\l[-c_1\delta^{-2/m}, c_1\delta^{-2/m}\r],\ E_j:=\l[-c_j\delta^{-\alpha_j/m}, c_j\delta^{-\alpha_j/m}\r] (2\leq j\leq n)$ are intervals and $c_j(1\leq j\leq n)$ are small positive constant we will decide later.
Then for a fixed point $x\in R_\delta$,
\begin{align*}
\l|T_{\alpha,m}f(x)\r|&=\l|\int_{\mathbb{R}^n}e^{i(x_1^{\alpha_1}y_1^m+\cdots+x_n^{\alpha_n}y_n^m)}\prod_{j=1}^nf_j(y_j)dy_1\cdots dy_n\r| \\
&=\prod_{j=1}^n\l|\int_{\mathbb{R}^n}e^{ix_j^{\alpha_j}y_j^m}f_j(y_j)dy_j\r|\\
&=\prod_{j=1}^n \l|F_j(x)\r|.
\end{align*}
Now we estimate each $\l|F_j(x)\r|$.\\

Case I: $j=1$.
\begin{align*}
\l|F_1(x)\r|&=\l|\int_{-\infty}^{+\infty}e^{ix_1^{\alpha_1}y_1^m}e^{-iy_1^m}\chi_{E_1}(y_1)dy_1\r| \\
&=\l|\int_{-c_1\delta^{-2/m}}^{+c_1\delta^{-2/m}}e^{i(x_1^{\alpha_1}-1)y_1^m}dy_1\r| \\
&=\l|\int_{-c_1\delta^{-2/m}}^{+c_1\delta^{-2/m}}e^{i(x_1-1)(x_1^{\alpha_1-1}+x_1^{\alpha_1-2}+\cdots+1)y_1^m}dy_1\r|.
\end{align*}
Provided that
\[\l|x_1-1\r|\leq c\delta^2,\quad |x_1|\leq 1+\delta^2<5/4,\quad |y_1|\leq c_1\delta^{-2/m},\] then
\[\l|(x_1^{\alpha_1}-1)y_1^m\r|=\l|(x_1-1)(x_1^{\alpha_1-1}+x_1^{\alpha_1-2}+\cdots+1)y_1^m\r|\leq \pi/3\]
if $c_1$ is small enough.\\
Therefore
\begin{equation*}
\l|F_1(x)\r| \geq \l|\int_{-c_1\delta^{-2/m}}^{+c_1\delta^{-2/m}}\cos\l((x_1^{\alpha_1}-1)y_1^m\r)dy_1\r|
\geq  c_1\delta^{-2/m}.
\end{equation*}

Case II: $2\leq j\leq n$.\\

In the same way, it follows that
\begin{equation*}
\l|F_j(x)\r|\geq c_j\delta^{-\alpha_j/m}.
\end{equation*}

In view of the arguments of Case I and Case II, we have
\begin{align*}
\int_{R_\delta}\l|T_{\alpha,m}(f)\r|^2dx &\geq \l(\prod_{j=1}^nc_j\r)\l|R_\delta\r|\l(\delta^{-2/m}\cdot\delta^{-\l(\sum_{j=2}^n\alpha_j\r)/m}\r)^2\\
&=\l(\prod_{j=1}^nc_j\r)\delta^{n+1}\l(\delta^{-2/m}\cdot\delta^{-\l(\sum_{j=2}^n\alpha_j\r)/m}\r)^2.
\end{align*}
On the other hand,
\begin{align*}
\|f\|_{L^p(\mathbb{R}^{n})}^p&=\int_{\mathbb{R}^n}\l|\prod_{j=1}^nf_j(y_j)\r|^pdy\\
&=\int_{\mathbb{R}^n}\l|\prod_{j=1}^n\chi_{E_j}(y_j)\r|^pdy\\
&\leq \l(\prod_{j=1}^nc_j\r)\l(\delta^{-2/m}\cdot\delta^{-\l(\sum_{j=2}^n\alpha_j\r)/m}\r).
\end{align*}
From (\ref{inequality-2}), we can obtain
\begin{equation*}
\delta^{n+1}\l(\delta^{-2/m}\cdot\delta^{-\l(\sum_{j=2}^n\alpha_j\r)/m}\r)^2\lesssim
\delta^2\l(\delta^{-2/m}\cdot\delta^{-\l(\sum_{j=2}^n\alpha_j\r)/m}\r)^{2/p},
\end{equation*}
and it equals
\begin{equation*}
\delta^{n+1-2\l(2+\sum_{j=1}^n\alpha_j\r)/m-2+2\l(2+\sum_{j=2}^n\alpha_j\r)/(pm)}\lesssim 1.
\end{equation*}
Let $\delta\to 0$, then the inequality above still holds and this requires
\begin{equation*}
n+1-2\l(2+\sum_{j=2}^n\alpha_j\r)/m-2+2\l(2+\sum_{j=2}^n\alpha_j\r)/(pm)\leq 0.
\end{equation*}
Thus the inequality is equivalent to
\begin{equation*}
p\leq \frac{2(2+\sum_{j=2}^n\alpha_j)}{2(2+\sum_{j=2}^n\alpha_j)-(n-1)m}.
\end{equation*}
The rectangle $R_\delta$ contained in the shell  $\{x:1-\delta^2\leq |x|\leq 1+\delta^2\}$ can be anyone of
\[R_\delta^j:=\{x=(x_1,\cdots,x_n):|x_j-1|\leq c\delta^2, |x_i|\leq c\delta, 1\leq i\neq j\leq n\}\quad (1\leq j\leq n).\]
Using the same arguments as in  Case I and Case II, we can easily carry out the necessary condition.
\end{proof}

\end{document}